\documentclass[11pt]{amsart}
\usepackage[margin=1in]{geometry}

\usepackage{amssymb}
\usepackage{amsthm}
\usepackage{amsmath}
\usepackage{mathrsfs}
\usepackage{amsbsy}
\usepackage{bm}
\usepackage{hyperref}
\usepackage{tikz}
\usepackage{array}
\usepackage{enumerate}
\usepackage{bbm}
\usepackage{comment}
\usepackage{mathtools}
\usepackage{makecell} 
\usepackage{colortbl}
\usepackage{xcolor}
\usepackage{tabularray} 
\usepackage{enumitem}

\DeclareFontFamily{U}{mathx}{}
\DeclareFontShape{U}{mathx}{m}{n}{<-> mathx10}{}
\DeclareSymbolFont{mathx}{U}{mathx}{m}{n}
\DeclareMathAccent{\widecheck}{0}{mathx}{"71}

\definecolor{lavender}{rgb}{0.4,0,1}
\definecolor{MyGreen}{rgb}{0,0.75,0}

\hypersetup{colorlinks=true, citecolor=lavender, linkcolor=lavender,urlcolor=lavender}

\usepackage{hhline}
\allowdisplaybreaks
\usepackage[noadjust]{cite}

\usepackage{caption}
\usepackage[noabbrev,capitalise,nameinlink]{cleveref}
\crefname{conjecture}{Conjecture}{Conjectures}

\newtheorem{theorem}{Theorem}[section]

\newtheorem{corollary}[theorem]{Corollary}

\newtheorem{lemma}[theorem]{Lemma}

\theoremstyle{definition}

\newtheorem{example}[theorem]{Example}

\newcommand{\Ech}{\mathrm{Ech}} 
\newcommand{\Cov}{\mathrm{Cov}} 
\newcommand{\PF}{\mathrm{PF}}
\newcommand{\Tree}{\mathrm{Tree}}
\newcommand{\oc}{\mathrm{oc}}
\newcommand{\des}{\mathrm{des}}
\newcommand{\Des}{\mathrm{Des}}
\newcommand{\exced}{\mathrm{exced}}
\newcommand{\cosum}{\mathrm{cosum}}

\newcommand{\rr}{\mathrm{rook}}
\newcommand{\Alt}{\mathrm{Alt}}

\newcommand{\rw}{\mathrm{rw}}
\newcommand{\RC}{\mathrm{RC}}

\newcommand{\dfn}[1]{\textcolor{blue}{\emph{#1}}}

\title{Short Proofs in Algebraic and Enumerative Combinatorics} 

\author{Colin Defant}

\begin{document}

\begin{abstract} 
We present several short proofs that resolve open problems from the algebraic and enumerative combinatorics literature. 
\begin{itemize}
\item First, we consider the echelonmotion operator on modular lattices. We resolve a conjecture of Defant, Jiang, Marczinzik, Segovia, Speyer, Thomas, and Williams and, consequently, obtain a new algebraic bijective proof of a classical result of Dilworth. 
\item Second, we consider statistics on parking functions studied by Stanley and Yin and by Hopkins. We prove some conjectures of Hopkins. 
\item Third, we consider centralizers in the plactic monoid. We settle two conjectures of Sagan and Wilson. 
\end{itemize} 
All of these proofs were obtained autonomously by ChatGPT~5.4 Pro. 
\end{abstract} 

\maketitle

\section{Introduction} 
The purpose of this article is to present several short proofs that settle open problems in algebraic and enumerative combinatorics. The inspiration to collect these proofs into a single manuscript came from recent articles by Alexeev, Putterman, Sawhney, Sellke, and Valiant at OpenAI \cite{APSSV1,APSSV2}. The format of those articles was, in turn, inspired by papers by Alon \cite{Alon1,Alon2,Alon3} and by Conlon, Fox, and Sudakov \cite{CFS1,CFS2}. Each section of this manuscript is devoted to a single topic and can be read independently of the other sections. 

All of the proofs in this article were discovered autonomously by ChatGPT~5.4 Pro.\footnote{The author was especially impressed by the clever constructions ChatGPT found to prove \cref{thm:echelonmotion,thm:fixed-content}.} The main contributions of the author were to find the problems, digest and verify the proofs, polish the writing, and provide exposition about each of the topics. The relevant ChatGPT conversations are included in an ancillary file \href{https://arxiv.org/abs/2605.19979}{here}.

\subsection{Notation and Terminology}
Let us briefly collect some notation and terminology that we will use throughout the manuscript. 

Given a proposition $\mathrm{P}$, let 
\[\mathbf{1}_{\mathrm{P}}=\begin{cases}1 & \text{if } \mathrm{P}\text{ is true}\\0 & \text{if }\mathrm{P}\text{ is false}.\end{cases}\]

Let $\mathbb N$ denote the set of positive integers. For $n\in\mathbb N$, let $[n]=\{1,\ldots,n\}$. Throughout this article, a \dfn{word} is a finite string of positive integers. For $X\subseteq\mathbb N$, we write $X^*$ for the set of words whose letters are in $X$.   

Let $\mathfrak S_n$ denote the symmetric group of permutations of $[n]$. A permutation $w\in\mathfrak S_n$ can be seen as a bijection $w\colon [n] \to [n]$ or as a permutation matrix whose $(i,j)$-entry is $\mathbf{1}_{w(i)=j}$. We also represent $w$ via its one-line notation $w(1)\cdots w(n)$, which is a word in $[n]^*$ in which each letter is used exactly once. 

We use English conventions when discussing arrays of boxes, numbering rows from top to bottom and numbering columns from left to right. 

\subsection{Outline}
\begin{itemize}[leftmargin=*]
\item \cref{sec:Echelonmotion} regards a bijection from a poset to itself known as \emph{echelonmotion}. This map depends on a choice of a linear extension of the poset. It is defined using the Bruhat decomposition of the general linear group, and it generalizes the classical rowmotion map. We prove a conjecture of Defant, Jiang, Marczinzik, Segovia, Speyer, Thomas, and Williams \cite{DefantEchelon} concerning echelonmotion on modular lattices (\cref{thm:echelonmotion}), which in turn yields bijective proofs of a classical result of Dilworth \cite{Dilworth}. 
\item \cref{sec:parking} regards certain statistics on parking functions. Some of the relevant generating functions are specializations of generating functions studied by Stanley and Yin \cite{StanleyYin}. We first prove a refined version of a conjecture of Hopkins about the equidistribution of some of these statistics (\cref{thm:fixed-content}). We then prove two additional enumerative conjectures of Hopkins that relate to simsun permutations and alternating permutations (\cref{thm:simsun,thm:alternating}).   
\item \cref{sec:plactic} regards the plactic monoid. Suppose $u$ is a word in $[m]^*$ whose RSK insertion tableau has $\ell$ rows. We prove that if $w$ is in the plactic centralizer of $u$, then all entries in the first $\ell$ rows of the insertion tableau of $w$ are at most $m$ (\cref{thm:plactic1}); this settles a conjecture of Sagan and Wilson \cite{SW25}. We then prove a different conjecture of Sagan and Wilson that relates the plactic centralizer of the $m$-reverse complement of $u$ to the $m$-evacuation of the plactic centralizer of $u$ (\cref{thm:6.5}).  
\end{itemize}  

\section{Echelonmotion on Modular Lattices}\label{sec:Echelonmotion}  
\dfn{Rowmotion} is a well-studied bijective operator on the set of order ideals of a finite poset. It appears in numerous settings, including dynamical algebraic combinatorics \cite{Pan09,SW12,AST13,Rob16,H24}, matroid theory \cite{DF90}, lattice theory \cite{Bar19,TW19,DW}, and representation theory \cite{IM22,MTY}. Recently, Kl\'asz, Marczinzik, and Thomas discovered a way to obtain rowmotion via the Bruhat decomposition of the general linear group \cite{KMT}. This provides a natural extension of the definition of rowmotion to an arbitrary finite poset, provided one fixes a linear extension of the poset in advance. In a subsequent article, Defant, Jiang, Marczinzik, Segovia, Speyer, Thomas, and Williams named this extension \emph{echelonmotion} \cite{DefantEchelon}, and they proved that special instances of echelonmotion recover other notable lattice-theoretic extensions of rowmotion.  

Let $R$ be a poset with $n$ elements. For $x,y\in R$, we say $y$ \dfn{covers} $x$, denoted $x\lessdot y$, if $x<y$ and there does not exist $z\in R$ such that $x<z<y$. A \dfn{linear extension} of $R$ is a bijection $\sigma\colon R\to[n]$ such that $\sigma(x)\leq\sigma(y)$ for all $x,y\in R$ with $x\leq y$. Define the \dfn{Cartan matrix} of $R$ with respect to a linear extension $\sigma$ to be the $n\times n$ matrix $W^{R,\sigma}$ whose $(i,j)$-entry is 
\[W_{i,j}^{R,\sigma}=\mathbf{1}_{\sigma^{-1}(i)\geq\sigma^{-1}(j)}.\] This matrix is lower-triangular with $1$'s on the diagonal, so it is invertible. The classical \dfn{Bruhat decomposition} of the general linear group $\mathrm{GL}_n(\mathbb C)$ states that 
\[\mathrm{GL}_n(\mathbb C)=\bigsqcup_{P\in\mathfrak S_n}BPB,\] where $B$ is the Borel subgroup of invertible upper-triangular complex $n\times n$ matrices and $\mathfrak S_n$ is the set of $n\times n$ permutation matrices. Because $W^{R,\sigma}$ is invertible, there is a unique permutation matrix $P^{R,\sigma}$ such that $W^{R,\sigma}\in BP^{R,\sigma}B$. This yields a bijection $\Ech_\sigma\colon R\to R$ defined so that $\Ech_\sigma(x)=y$ if and only if $P^{R,\sigma}_{\sigma(y),\sigma(x)}=1$. 

\begin{figure}[ht]
  \begin{center}
  \includegraphics[height=5.8cm]{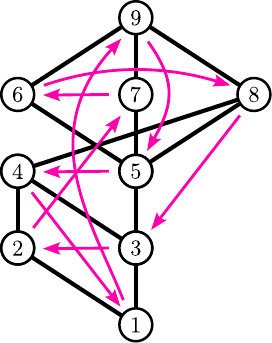}\qquad\qquad\qquad\includegraphics[height=5.8cm]{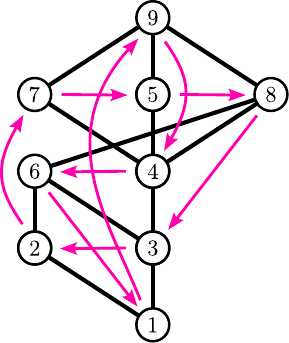}
  \end{center}
\caption{A modular lattice labeled by two different linear extensions. The pink arrows represent echelonmotion with respect to the given linear extension. }\label{fig:Echelon} 
\end{figure}

When $R$ is a distributive lattice, Kl\'asz, Marczinzik, and Thomas proved that $\Ech_\sigma$ coincides with the classical rowmotion operator on $R$ \cite{KMT}. In particular, $\Ech_\sigma$ is independent of $\sigma$ when $R$ is distributive. This motivated Defant, Jiang, Marczinzik, Segovia, Speyer, Thomas, and Williams to define a finite poset $R$ to be \dfn{echelon-independent} if $\Ech_\sigma=\Ech_{\sigma'}$ for all linear extensions $\sigma,\sigma'$ of $R$. Among other results, they proved that a finite lattice is echelon-independent if and only if it is semidistributive \cite{DefantEchelon}. Kl\'asz, Kleinau, and Marczinzik subsequently proved that a poset is echelon-independent if its incidence algebra is Auslander regular \cite{KKM}.   

Now assume $R$ is a finite lattice with meet operation $\wedge$ and join operation $\vee$. We say $R$ is \dfn{modular} if for all $a,b,x\in R$ with $a\leq b$, we have $a\vee (x\wedge b)=(a\vee x)\wedge b$. It is well known that $R$ is modular if and only if for all $a,b\in R$, we have \[a\wedge b\lessdot a\quad \Longleftrightarrow\quad b\lessdot a\vee b.\] Modular lattices arise frequently in abstract algebra. For example, the lattice of normal subgroups of a finite group is modular. 

For $x\in R$, let \[\Cov_R^\downarrow(x)=\{y\in R:y\lessdot x\}\quad\text{and}\quad\Cov_R^\uparrow(x)=\{y\in R:x\lessdot y\}\] be the set of elements covered by $x$ and the set of elements covering $x$, respectively. A foundational result due to Dilworth \cite{Dilworth} states that if $R$ is a finite modular lattice, then \[\left|\{x\in R:|\Cov_R^\downarrow(x)|=k\}\right|=\left|\{x\in R:|\Cov_R^\uparrow(x)|=k\}\right|\] for every integer $k$. 

The following theorem, which was conjectured by Defant, Jiang, Marczinzik, Segovia, Speyer, Thomas, and Williams \cite{DefantEchelon}, yields a new linear-algebraic proof of Dilworth's theorem. In fact, it shows that for each linear extension $\sigma$ of a finite modular lattice $R$, echelonmotion with respect to $\sigma$ provides a bijective proof of Dilworth's theorem for $R$.   

\begin{theorem}\label{thm:echelonmotion} 
Let $\sigma$ be a linear extension of a finite modular lattice $R$. For every $x\in R$, we have \[|\Cov_R^\uparrow(\Ech_\sigma(x))|=|\Cov_R^\downarrow(x)|.\]  
\end{theorem} 
\begin{proof}
Let $n=|R|$, and let $W=W^{R,\sigma}$. Define two $n\times n$ matrices $A$ and $B$ by letting \[A_{\sigma(y),\sigma(z)}=\begin{cases}|\Cov_R^\uparrow(y)| & \text{if } z=y\\-1 & \text{if } y\lessdot z\\0 & \text{otherwise}\end{cases}\qquad\text{and}\qquad B_{\sigma(w),\sigma(x)}=\begin{cases}|\Cov_R^\downarrow(x)| & \text{if } w=x\\-1 & \text{if } w\lessdot x\\0 & \text{otherwise}\end{cases}\] for all $w,x,y,z\in R$. Since $\sigma$ is a linear extension, both $A$ and $B$ are upper-triangular. We claim that $AW=WB$. Indeed, for $x,y\in R$, the $(\sigma(y),\sigma(x))$-entry of $AW$ is 
\begin{equation}\label{eq:1.1}(AW)_{\sigma(y),\sigma(x)}=|\Cov_R^\uparrow(y)|\cdot\mathbf{1}_{x\le y}-|\{z\in R: y\lessdot z,\ x\le z\}|,
\end{equation} whereas the $(\sigma(y),\sigma(x))$-entry of $WB$ is 
\begin{equation}\label{eq:1.2} (WB)_{\sigma(y),\sigma(x)}=|\Cov_R^\downarrow(x)|\cdot\mathbf{1}_{x\le y}-|\{w\in R: w\lessdot x,\ w\le y\}|.
\end{equation} If $x\le y$, then every element covering $y$ lies above $x$, and every element covered by $x$ lies below $y$. Hence, the expressions in \eqref{eq:1.1} and \eqref{eq:1.2} are both equal to $0$ when $x\leq y$. 

Now suppose $x\nleq y$ so that $\mathbf{1}_{x\leq y}=0$. If $z$ covers $y$ and $x\le z$, then \[y < x\vee y \le z,\] so $z=x\vee y$. Therefore, \[|\{z\in R: y\lessdot z,\ x\le z\}|=\mathbf{1}_{y\lessdot x\vee y}.\] Similarly, if $w\lessdot x$ and $w\le y$, then \[w\le x\wedge y < x,\] so $w=x\wedge y$. Therefore, \[|\{w\in R: w\lessdot x,\ w\le y\}|=\mathbf{1}_{x\wedge y\lessdot x}.\] Because $R$ is modular, we have $x\wedge y\lessdot x$ if and only if $y\lessdot x\vee y$. It follows that the expressions in \eqref{eq:1.1} and \eqref{eq:1.2} are equal when $x\not\leq y$. 

We have proven the claim that $AW=WB$. Now choose invertible upper-triangular $n\times n$ matrices $U$ and $V$ such that $W=UPV$, where $P=P^{R,\sigma}$. Using the equation $AW=WB$, we obtain that $AUPV=UPVB$. Thus, $(U^{-1}AU)P=P(VBV^{-1})$. Set \[\widetilde A=U^{-1}AU\quad\text{and}\quad\widetilde B=VBV^{-1}.\] The matrices $\widetilde A$ and $\widetilde B$ are upper-triangular. Moreover, conjugating an upper-triangular matrix by an invertible upper-triangular matrix preserves diagonal entries, so \[\widetilde A_{\sigma(u),\sigma(u)}=|\Cov_R^\uparrow(u)|\quad\text{and}\quad\widetilde B_{\sigma(u),\sigma(u)}=|\Cov_R^\downarrow(u)|\] for all $u\in R$. The equation $\widetilde A P=P\widetilde B$ gives \[P^{-1}\widetilde A P=\widetilde B.\] Let $x\in R$, and set $y=\Ech_{\sigma}(x)$. By the definition of echelonmotion, $P_{\sigma(y),\sigma(x)}=1$. Taking the $(\sigma(x),\sigma(x))$-entry of $P^{-1}\widetilde A P=\widetilde B$, we obtain \[|\Cov_R^\downarrow(x)|=\widetilde B_{\sigma(x),\sigma(x)}=\widetilde A_{\sigma(y),\sigma(y)}=|\Cov_R^\uparrow(y)|=|{\Cov}^{\uparrow}_{R}({\Ech}_{\sigma}(x))|,\] as desired. 
\end{proof}

\section{Parking Function Statistics}\label{sec:parking} 
\subsection{Generating Functions} 
Write $\Tree(n+1)$ for the set of trees with vertex set $\{0,1,\ldots,n\}$, rooted at $0$.  An \dfn{inversion} of a tree $T\in \Tree(n+1)$ is a pair $(i,j)$ with $1\leq i<j\leq n$ such that $j$ lies on the path from $i$ to the root.  Let $\mathrm{inv}(T)$ denote the number of inversions of $T$. The \dfn{tree inversion enumerator} is
\[
I_n(q)=\sum_{T\in \Tree(n+1)}q^{\mathrm{inv}(T)}.
\]

A \dfn{parking function} of length $n$ is a sequence $\pi=(\pi_1,\ldots,\pi_n)$ of positive integers such that if $b_1\leq \cdots \leq b_n$ is the weakly increasing rearrangement of the entries of $\pi$, then $b_j\leq j$ for all $j$.  Let $\PF(n)$ be the set of parking functions of length $n$.  For $\pi\in\PF(n)$, define
\[
\cosum(\pi)=\binom{n+1}{2}-\sum_{i=1}^n \pi_i.
\]
Kreweras \cite{Kreweras} proved that 
\[I_n(q)=\sum_{\pi\in \PF(n)}q^{\cosum(\pi)}.
\]

Hopkins introduced two $t$-analogues of $I_n(q)$ \cite{Hopkins}.  The analogue on the tree side is
\[
I_n(q,t)=\sum_{T\in \Tree(n+1)}q^{\mathrm{inv}(T)}t^{\mathrm{lev}(T)-1},
\]
where $\mathrm{lev}(T)$ is the number of leaves of $T$.  The analogue on the parking function side is
\[\widetilde I_n(q,t)=\sum_{\pi\in \PF(n)}q^{\cosum(\pi)}t^{\exced(\pi)},
\]
where
\[
\exced(\pi)=|\{i\in [n]:\pi_i>i\}|
\]
is the number of \dfn{strict excedances} of $\pi$.

To compare these two analogues, one uses the ordinary parking procedure.  A line of $n$ cars arrives at a parking lot with $n$ parking spots numbered $1,\ldots,n$ from left to right.  The $i$-th car first tries to park in its preferred parking spot, which is spot $\pi_i$; if it is occupied, the car moves forward to the next available spot.  The sequence $\pi$ is a parking function precisely when all cars park successfully.  Define the \dfn{parking outcome}
\[
\oc(\pi)=\sigma(1)\cdots \sigma(n)\in \mathfrak S_n
\]
by declaring that the $i$-th car parks in spot $\sigma(i)$.  Thus, $\oc(\pi)^{-1}$ lists, from left to right along the street, which car occupies each spot.

A \dfn{descent} of a permutation $w\in\mathfrak S_n$ is an index $i\in[n-1]$ such that ${w(i)>w(i+1)}$. Let $\Des(w)=\{i\in[n-1]:w(i)>w(i+1)\}$ be the set of descents of $w$, and let $\des(w)=|\Des(w)|$. 
Stanley and Yin \cite{StanleyYin} studied a four-variable parking-function/tree generating function, and one of its two-variable specializations yields the identity
\[
I_n(q,t)=\sum_{\pi\in \PF(n)}q^{\cosum(\pi)}t^{\des(\oc(\pi)^{-1})}.
\]
This motivated Hopkins \cite{Hopkins} to conjecture the identity \[
\widetilde I_n(q,t)=\sum_{\pi\in \PF(n)}q^{\cosum(\pi)}t^{\des(\oc(\pi))}.
\]  

A \dfn{parking content} is a sequence $b=(b_1,\ldots, b_n)$ of positive integers such that $b_1\leq\cdots\leq b_n$ and such that $b_j\leq j$ for all $j$.  Let $\PF_{\leq}(n)$ be the set of parking contents of length $n$. Given a parking content $b\in\PF_{\leq}(n)$ and a permutation $w=w_1\cdots w_n\in \mathfrak S_n$, we define \[\pi^{b,w}=(\pi_1^{b,w},\ldots,\pi_n^{b,w})\in\PF(n),\] where $\pi^{b,w}_i=b_{w(i)}$.  We also define $\sigma^{b,w}=\oc(\pi^{b,w})\in\mathfrak S_n$. 

Our main result in this section is the following theorem. 

\begin{theorem}\label{thm:fixed-content}
For every fixed parking content $b\in\PF_{\leq}(n)$, we have 
\[
\sum_{w\in \mathfrak S_n}t^{\exced(\pi^{{b,w}})}=\sum_{w\in \mathfrak S_n}t^{\mathrm{des}(\sigma^{b,w})}.
\]
\end{theorem}

\cref{thm:fixed-content} implies Hopkins's conjecture. 

\begin{corollary}\label{cor:Hopkins}
For each $n\geq 1$, we have \[
\widetilde I_n(q,t)=\sum_{\pi\in \PF(n)}q^{\cosum(\pi)}t^{\des(\oc(\pi))}.
\]
\end{corollary}
\begin{proof}
For $b\in \operatorname{PF}_{\le}(n)$, let $\PF_b$ be the set of parking functions in $\PF(n)$ whose weakly increasing rearrangement is $b$. Also, define
\[
\mu(b)=\prod_{a\ge 1} |\{i\in[n]: b_i=a\}|!.
\]
We claim that for each $\pi\in\PF_b$, we have 
\[\mu(b)=|\{w\in\mathfrak S_n:\pi^{b,w}=\pi\}|.\]
Indeed, for each positive integer $a$, one must biject the indices $i$ with $\pi_i=a$ with the indices $j$ such
that $b_j=a$; these choices are independent for the different values of $a$. It follows that
\[
\sum_{\pi\in \operatorname{PF}_b} t^{\operatorname{exced}(\pi)}
=
\frac{1}{\mu(b)}
\sum_{w\in \mathfrak S_n} t^{\operatorname{exced}(\pi^{b,w})}
\]
and, since $\sigma^{b,w}=\operatorname{oc}(\pi^{b,w})$,
\[
\sum_{\pi\in \operatorname{PF}_b}
t^{\operatorname{des}(\operatorname{oc}(\pi))}
=
\frac{1}{\mu(b)}
\sum_{w\in \mathfrak S_n} t^{\operatorname{des}(\sigma^{b,w})}.
\]
Moreover, the statistic $\operatorname{cosum}$ is constant on
$\operatorname{PF}_b$. More explicitly, 
\[
\operatorname{cosum}(\pi)=\binom{n+1}{2}-\sum_{i=1}^n b_i
\]
for every $\pi\in\PF_b$. 

Multiplying the identity in \cref{thm:fixed-content} by
\[
\frac{q^{\binom{n+1}{2}-\sum_{i=1}^n b_i}}{\mu(b)}
\]
therefore gives
\[
\sum_{\pi\in \operatorname{PF}_b}
q^{\operatorname{cosum}(\pi)}t^{\operatorname{exced}(\pi)}
=
\sum_{\pi\in \operatorname{PF}_b}
q^{\operatorname{cosum}(\pi)}
t^{\operatorname{des}(\operatorname{oc}(\pi))}.
\]
Summing this identity over all $b\in \operatorname{PF}_{\le}(n)$ yields
\[
\widetilde I_n(q,t)
=
\sum_{\pi\in \operatorname{PF}(n)}
q^{\operatorname{cosum}(\pi)}
t^{\operatorname{exced}(\pi)}=
\sum_{\pi\in \operatorname{PF}(n)}
q^{\operatorname{cosum}(\pi)}
t^{\operatorname{des}(\operatorname{oc}(\pi))}. \qedhere 
\]
\end{proof}

The proof of \cref{thm:fixed-content} is naturally expressed using a Ferrers board.  Fix $b\in\PF_\leq(n)$, and let $B_b$ be the board with columns $1,\ldots,n$ and rows $1,\ldots,n$, where column $c$ contains the cells
\[
(1,c),(2,c),\ldots,(b_c-1,c).
\]
Equivalently, $(r,c)\in B_b$ if and only if $r<b_c$.  Let $\rr_k(B_b)$ denote the number of ways to place $k$ nonattacking rooks on $B_b$.

\begin{lemma}\label{lem:excedance-side}
For each fixed $b\in\PF_\leq(n)$, we have \[\sum_{w\in \mathfrak S_n}t^{\exced(\pi^{b,w})}=\sum_{k\geq 0}\rr_k(B_b)(n-k)!(t-1)^k.\]
\end{lemma}

\begin{proof}
A permutation $w\in \mathfrak S_n$ is equivalent to a full nonattacking rook placement on an $n\times n$ board with one rook in each row and each column. Indeed, we simply put the rook in column $w(i)$ in row $i$; this rook lies in $B_b$ exactly when $i<b_{w(i)}$, which is exactly when $i$ is a strict excedance of $\pi^{b,w}$. 

Expanding $t^m=(1+(t-1))^m$, we find that the coefficient of $(t-1)^k$ in $\sum_{w\in \mathfrak S_n}t^{\exced(\pi^{b,w})}$ is the number of full nonattacking  rook placements on an $n\times n$ board together with $k$ marked rooks lying in $B_b$.  There are $\rr_k(B_b)$ ways to choose the $k$ marked rooks first and then $(n-k)!$ ways to biject the remaining $n-k$ rows with the remaining $n-k$ columns.
\end{proof}

Note that the coefficient of $(t-1)^k$ in $\sum_{w\in \mathfrak S_n}t^{\mathrm{des}(\sigma^{b,w})}$ is the number of pairs $(w,A)$ with $w\in\mathfrak S_n$ such that $A\subseteq \Des(\sigma^{b,w})$ and $|A|=k$. For such a pair, define
\[\Phi(w,A)=\{(\sigma^{b,w}(i+1),w(i)):i\in A\}.
\]

\begin{lemma}\label{lem:Phi-board}
Fix $b\in\PF_\leq(n)$. Suppose $w\in \mathfrak S_n$ and $A\subseteq\Des(\sigma^{b,w})$. Then $\Phi(w,A)$ is a nonattacking $|A|$-rook placement on $B_b$. 
\end{lemma}

\begin{proof}
Because $\sigma^{b,w}$ and $w$ are permutations, $\Phi(w,A)$ is certainly a $|A|$-rook placement on an $n\times n$ board. It remains to show that for each $i\in A$, the cell $(\sigma^{b,w}(i+1),w(i))$ is in $B_b$.

Fix $i\in A$.  Note that the $i$-th car parked in spot $\sigma^{b,w}(i)$, the $(i+1)$-th car parked in spot $\sigma^{b,w}(i+1)$, and
$\sigma^{b,w}(i)>\sigma^{b,w}(i+1)$.
When the $i$-th car arrived, the later spot $\sigma^{b,w}(i+1)$ was still empty.  If $b_{w(i)}\leq \sigma^{b,w}(i+1)$, then the $i$-th car, starting at spot $b_{w(i)}$ and moving forward, would have parked at some empty spot at or before $\sigma^{b,w}(i+1)$, contradicting the fact that $\sigma^{b,w}(i)>\sigma^{b,w}(i+1)$.  This shows that
$\sigma^{b,w}(i+1)<b_{w(i)}$, so $(\sigma^{b,w}(i+1),w(i))\in B_b$.
\end{proof}

\begin{proof}[Proof of \cref{thm:fixed-content}]
We claim that every nonattacking $k$-rook placement on $B_b$ has exactly ${(n-k)!}$ preimages under~$\Phi$. If we can prove this claim, then the desired identity will follow immediately from \cref{lem:excedance-side,lem:Phi-board}. 

Let $R=\{(r_1,c_1),\ldots,(r_k,c_k)\}$ be a nonattacking $k$-rook placement on $B_b$, where ${r_1<\cdots<r_k}$. We describe a bijection between $\Phi^{-1}(R)$ and the set of words obtained by ordering $[n]\setminus\{c_1,\ldots,c_k\}$. 

Start with a word $u^{(0)}$ obtained by ordering $[n]\setminus\{c_1,\ldots,c_k\}$.  We will generate a sequence $u^{(0)},u^{(1)},\ldots,u^{(k)}$ of words, where each $u^{(j)}$ is obtained by inserting the number $c_j$ into $u^{(j-1)}$. Suppose we have already constructed $u^{(j-1)}$.  Think of the numbers in $u^{(j-1)}$ as labels of cars, and park them on the street with spots $1,\ldots,n$ via the standard parking procedure, where the car labeled $d$ has preference~$b_d$. (The label of a car is not the same as its arrival time.) 

We first claim that spot $r_{j}$ is occupied.  The columns missing from $u^{(j-1)}$ are precisely the columns of the rooks in $R$ whose rows are at least $r_j$.  If such a missing column is $c$ and its rook is $(r,c)$ with $r\geq r_j$, then $(r,c)\in B_b$, so $b_{c}>r\geq r_j$.  Thus, every column $d$ with $b_d\leq r_j$ is present in $u^{(j-1)}$.  Since $b$ is a parking content, we have $b_m\leq m$ for all $1\leq m\leq r_j$, so there are at least $r_j$ cars $d$ satisfying $b_d\leq r_j$.  One of these cars must occupy spot $r_j$. Let $d_j$ be the label of the car currently parked in spot $r_j$.  Let us construct $u^{(j)}$ by inserting $c_j$ immediately before $d_j$ in the word $u^{(j-1)}$.  Because $(r_j,c_j)\in B_b$, we have $b_{c_j}>r_j$.  Therefore, when we perform the parking procedure to $u^{(j)}$ (where each car labeled $d$ has preference $b_d$), the car labeled $c_j$ cannot park at or before spot $r_j$.  Hence, the car labeled $d_{j}$ still parks in spot $r_j$, and the car labeled $c_j$ parks to the right of the car labeled $d_j$.  

At the end of this process, we obtain a permutation $u^{(k)}\in\mathfrak S_n$. By construction, for each $1\leq j\leq k$, the number $c_j$ appears immediately before the number $d_j$ in $u^{(k)}$; say $c_j=u^{(k)}(i_j)$ so that $d_j=u^{(k)}(i_j+1)$. When we apply the parking procedure with each car labeled $d$ having preference $b_d$, the car labeled $d_j$ parks in spot $r_j$. In other words, $\sigma^{b,u^{(k)}}(i_j+1)=r_j$. Note that $i_j$ is a descent of $\sigma^{b,u^{(k)}}$ because the car labeled $c_j$ parks to the right of the car labeled $d_j$. Let $A=\{i_1,\ldots,i_k\}$.  Then $A\subseteq \Des(\sigma^{b,u^{(k)}})$, and $\Phi(u^{(k)},A)=R$.  This constructs one element of $\Phi^{-1}(R)$ from the initial word $u^{(0)}$. Consequently, $|\Phi^{-1}(R)|\geq (n-k)!$. 

The above insertion process is reversible. Indeed, let
\((\widehat w,\widehat A)\in \Phi^{-1}(R)\). For each \(j\), there is a unique
index \(i_j\in \widehat A\) such that
\[
(\sigma^{b,\widehat w}(i_j+1),\widehat w(i_j))=(r_j,c_j).
\]
Set \(d_j=\widehat w(i_j+1)\). Thus, \(c_j\) is immediately followed by \(d_j\)
in \(\widehat w\), and the car labeled \(d_j\) parks in spot \(r_j\). We recover the initial word by deleting \(c_k,c_{k-1},\ldots,c_1\) from $\widehat w$, in this
order. We claim that, just before \(c_j\) is deleted, deleting \(c_j\) is the
inverse of the \(j\)-th insertion step described above. The only deletions that
have already been performed are those of \(c_\ell\) with \(\ell>j\). Since
\((r_\ell,c_\ell)\in B_b\), we have \(b_{c_\ell}>r_\ell>r_j\), so these cars
cannot affect the occupancy of spots at or below \(r_j\) before the car \(d_j\)
arrives. Hence, after those earlier deletions, the car \(d_j\) still parks in
spot \(r_j\). Moreover, \(b_{c_j}>r_j\), and since
\(i_j\in \widehat A\subseteq \operatorname{Des}(\sigma^{b,\widehat w})\), the
car \(c_j\) parks to the right of spot \(r_j\). Therefore, inserting \(c_j\)
immediately before \(d_j\) leaves \(d_j\) parked in spot \(r_j\), exactly as in
the forward construction. This proves reversibility. 

This completes the proof of the claim and, hence, the theorem. 
\end{proof} 

\begin{figure}[ht]
  \begin{center}
  \includegraphics[height=3.5cm]{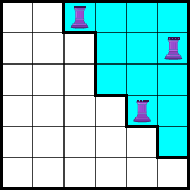}
  \end{center}
\caption{The rook placement from \cref{exam:Hopkins}.  }\label{fig:rook} 
\end{figure}

\begin{example}\label{exam:Hopkins}
    Let $n=6$ and $k=3$. Fix $b=(1,1,2,4,5,6)\in\PF_{\leq}(6)$. The $3$-rook placement $R=\{(1,3),(2,6),(4,5)\}$, with the board $B_b$ shaded in cyan, appears in \cref{fig:rook}.  
Let us choose the word $u^{(0)}=241$. The parking procedure produces the arrangement of cars \[\begin{array}{l}
    \includegraphics[height=1cm]{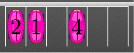}\end{array}.\]
The spot $r_1=1$ is occupied by the car labeled $2$, so we obtain $u^{(1)}$ by inserting the number $c_1=3$ into $u^{(0)}$ immediately before $2$. Hence, $u^{(1)}=3241$. 
Now the parking procedure produces the arrangement of cars \[\begin{array}{l}
    \includegraphics[height=1cm]{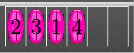}\end{array}.\]
The spot $r_2=2$ is occupied by the car labeled $3$, so we obtain $u^{(2)}$ by inserting the number $c_2=6$ into $u^{(1)}$ immediately before $3$. Hence, $u^{(2)}=63241$. 
Now the parking procedure produces the arrangement of cars \[\begin{array}{l}
    \includegraphics[height=1cm]{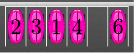}\end{array}.\]
The spot $r_3=4$ is occupied by the car labeled $4$, so we obtain $u^{(3)}$ by inserting the number $c_3=5$ into $u^{(2)}$ immediately before $4$. Hence, $u^{(3)}=632541$. Now the parking procedure produces the arrangement of cars \[\begin{array}{l}
    \includegraphics[height=1cm]{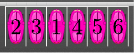}\end{array}.\] 
We compute that $i_1=2$, $i_2=1$, and $i_3=4$, so $A=\{1,2,4\}$. In addition, $\pi^{b,u^{(3)}}=(6,2,1,5,4,1)$, so $\sigma^{b,u^{(3)}}=621543$. As expected, we have $A\subseteq\Des(\sigma^{b,u^{(3)}})$, and $\Phi(u^{(3)},A)=R$. 
\end{example}

\subsection{Specializations} 
We now turn our attention to the specializations of $I_n(q,t)$ and $\widetilde I_n(q,t)$ at $q=-1$. 

Let $u=u_1\cdots u_m$ be a word. A \dfn{double descent} of $u$ is an index $i\in\{2,\ldots,m-1\}$ such that $u_{i-1}>u_i>u_{i+1}$. A permutation $w\in\mathfrak S_n$ is \dfn{simsun} if for every $j$, the word obtained from $w$ by deleting all entries greater than $j$ has no double descents. Let $\mathrm{Simsun}(n)$ denote the set of simsun permutations in
$\mathfrak S_n$. 

The following theorem resolves a conjecture of Hopkins \cite{Hopkins}. 

\begin{theorem}\label{thm:simsun}
For $n\geq 1$, we have 
\[
        I_n(-1,t)
        =
        \sum_{w\in\operatorname{Simsun}(n-1)}
        t^{n-1-\operatorname{des}(w)}.
\]
\end{theorem}

\begin{proof}
We make the convention that $I_0(-1,t)=1$. Let $[m]_q=\frac{1-q^m}{1-q}=1+q+\cdots+q^{m-1}$. Note that $[m]_{-1}=\mathbf{1}_{m\equiv 1\!\!\pmod 2}$.
Stanley and Yin \cite{StanleyYin} proved that \[
I_n(q,t)
=
[n]_q I_{n-1}(q,t)
+
t\sum_{i=0}^{n-2}
\binom{n-1}{i}
[i+1]_q I_i(q,t)I_{n-1-i}(q,t).
\]
This specializes to
\begin{equation}
I_n(-1,t)
=
\mathbf 1_{n\equiv 1\!\!\!\!\!\!\pmod 2}I_{n-1}(-1,t)
+
t\sum_{\substack{0\leq i\leq n-2\\ i\text{ even}}}
\binom{n-1}{i}I_i(-1,t)I_{n-1-i}(-1,t).
\label{eq:minus-one-A-rec}
\end{equation}

Let
\[
        E(z):=\sum_{n\geq 0} I_n(-1,t)\frac{z^n}{n!},
        \qquad
        E_+(z):=\frac{E(z)+E(-z)}2,
        \qquad
        E_-(z):=\frac{E(z)-E(-z)}2.
\]
Multiplying \eqref{eq:minus-one-A-rec} by $z^{n-1}/(n-1)!$ and summing
over $n\geq 1$ gives
\[
        E'(z)=E_+(z)\bigl(1-t+tE(z)\bigr).
\]
Equivalently,
\[
        E_+'(z)=tE_+(z)E_-(z),
        \qquad
        E_-'(z)=E_+(z)(1-t+tE_+(z)),
\]
with $E_+(0)=1$ and $E_-(0)=0$.  Set
\[
        W(z):=1-t+tE_+(z).
\]
Then
\[
        W'(z)=tE_+'(z)=t^2E_+(z)E_-(z)=t(W(z)-1+t)E_-(z),
\]
and
\[
        E_-'(z)=E_+(z)W(z)=\frac{W(z)-1+t}{t}W(z).
\]
Thus,
\[
        \frac{d}{dz}\bigl(W(z)^2-t^2E_-(z)^2\bigr)=0.
\]
Since $W(0)=1$ and $E_-(0)=0$, we obtain
\[
        W(z)^2-t^2E_-(z)^2=1.
\]
Write $W(z)=\cosh\theta$ and $tE_-(z)=\sinh\theta$.  Then
\[
        \theta'=t-1+\cosh\theta,
        \qquad
        \theta(0)=0.
\]
With $y=\tanh(\theta/2)$, this becomes
\[
        2\frac{dy}{dz}=t+(2-t)y^2,
\]
so
\[
        y=
        \sqrt{\frac{t}{2-t}}
        \tan\!\left(\frac{z\sqrt{t(2-t)}}2\right).
\]
Finally,
\begin{equation}
        E(z)
        =
        E_+(z)+E_-(z)
        =
        \frac{e^\theta+t-1}{t}
        =
        \frac{t+(2-t)y}{t(1-y)}.
\label{eq:minus-one-EGF}
\end{equation}

Now let
\[
        R_m(x):=
        \sum_{w\in\mathrm{Simsun}(m)}
        x^{\mathrm{des}(w)}.
\]
According to \cite[Theorem~2.1]{ChowShiu}, we have
\begin{equation}
        R_{m+1}(x)
        =
        (1+mx)R_m(x)+x(1-2x)R_m'(x),
        \qquad
        R_0(x)=1.
\label{eq:simsun-rec}
\end{equation}
Define
\[
        A_n(t):=t^{n-1}R_{n-1}(1/t).
\]
Transforming \eqref{eq:simsun-rec} gives
\begin{equation}
        A_{n+1}(t)
        =
        \bigl(1+n(t-1)\bigr)A_n(t)+t(2-t)A_n'(t),
        \qquad
        A_1(t)=1.
\label{eq:B-rec}
\end{equation}
On the other hand, it follows from the closed-form expression in \eqref{eq:minus-one-EGF} that
\[
        \frac{\partial}{\partial z}E(z)
        =
        E(z)+(t-1)z\frac{\partial}{\partial z}E(z)+t(2-t)\frac{\partial}{\partial t}E(z).
\]
Extracting the coefficient of $z^n/n!$ gives exactly the recurrence
\eqref{eq:B-rec} for the coefficients $I_n(-1,t)$, with the same initial
condition.  Therefore,
\[
        I_n(-1,t)=A_n(t)
        =
        t^{n-1}R_{n-1}(1/t)=\sum_{w\in\mathrm{Simsun}(n-1)}
        t^{n-1-\mathrm{des}(w)},
\]
as desired. 
\end{proof}

A \dfn{big descent} of a permutation $w\in\mathfrak S_n$ is an index $i\in[n-1]$ such that $w(i)>w(i+1)+1$. Let $\des_1(w)$ denote the number of big descents of $w$. We say $w$ is \dfn{up-down alternating} if \[w_1<w_2>w_3<\cdots.\] Let $\mathrm{Alt}_n$ be the set of up-down alternating permutations in $\mathfrak S_n$.

The next theorem resolves yet another conjecture of Hopkins \cite{Hopkins}. 

\begin{theorem}\label{thm:alternating} 
For $n\geq 2$, we have 
\[
        \widetilde I_n(-1,t)
        =
        \sum_{w\in \Alt_n}t^{\mathrm{des}_1(w^{-1})+1}.
\]
\end{theorem}
\begin{proof}
By \cref{cor:Hopkins}, we know that 
\[
        \widetilde I_n(-1,t)
        =
        \sum_{\pi\in \PF(n)}
        (-1)^{\operatorname{cosum}(\pi)}
        t^{\operatorname{des}(\operatorname{oc}(\pi))}.
\]
We group the terms in this sum by parking outcome.  For
$\sigma\in \mathfrak S_n$, let \[
        \ell_i(\sigma)
        :=
        1+\max\{r\in\mathbb N_0\setminus\{\sigma(1),\ldots,\sigma(i-1)\}:r<\sigma(i)\},
\]
where $\mathbb N_0=\mathbb N\cup\{0\}$. A parking function $\pi=(\pi_1,\ldots,\pi_n)\in\PF(n)$ has parking outcome $\oc(\pi)=\sigma$ if and only if $\ell_i(\sigma)\leq \pi_i\leq \sigma(i)$ for all $1\leq i\leq n$. Therefore,
\[
\sum_{\substack{\pi\in\PF(n) \\ \operatorname{oc}(\pi)=\sigma}}
(-1)^{\operatorname{cosum}(\pi)}
=
(-1)^{\binom{n+1}{2}}
\prod_{i=1}^n
\sum_{a=\ell_i(\sigma)}^{\sigma(i)}(-1)^a.
\]
The right-hand side is $0$ unless
$\sigma(i)\equiv\ell_i(\sigma)\pmod 2$ for all $i$, in which case it is
\[
        (-1)^{\binom{n+1}{2}}
        (-1)^{\sigma(1)+\cdots+\sigma(n)}
        =
        1.
\]
Thus,
\begin{equation}
        \widetilde I_n(-1,t)
        =
        \sum_{\sigma\in\mathfrak O_n}
        t^{\operatorname{des}(\sigma)},
\label{eq:O-sum}
\end{equation}
where $\mathfrak O_n$ is the set of permutations $\sigma\in \mathfrak S_n$ such that $\sigma(i)\equiv\ell_i(\sigma)\pmod 2$ for every $i$.

For $p\in[n]$ and $\tau\in\mathfrak S_n$, define
\[
        \alpha_{\tau}(p):=
        \max(\{j\in[n]:j<p,\,\tau(j)>\tau(p)\}\cup\{0\}).
\]
Let $\mathfrak T_n$ be the set of permutations in $\mathfrak S_n$ such that $p-\alpha_\tau(p)$ is odd for every $p\in[n]$. If $p=\sigma(i)$, then
$\ell_i(\sigma)=\alpha_{\sigma^{-1}}(p)+1$.
It follows that $\mathfrak T_n=\{\sigma^{-1}:\sigma\in\mathfrak O_n\}$. Then \eqref{eq:O-sum} becomes
\begin{equation}
        \widetilde I_n(-1,t)
        =
        \sum_{\tau\in\mathfrak T_n}
        t^{\operatorname{des}(\tau^{-1})}.
\label{eq:T-ret}
\end{equation}

The class $\mathfrak T=\bigcup_{n\geq 0}\mathfrak T_n$ has a simple recursive description.  Suppose $\tau\in\mathfrak T_n$. If $n=\tau(p)$, then $\alpha_\tau(p)=0$, so $p$ must be odd.  Thus, the
subword to the left of $n$ has even length.  Moreover, deleting $n$ splits
$\tau$ into left and right
subwords whose standardizations belong to $\mathfrak T$.  Conversely, if $\tau=u\,n\,v$ is such that the left subword $u$ has even length and the standardizations of $u$ and $v$
belong to $\mathfrak T$, then 
$\tau\in\mathfrak T_n$. Let $w_\circ$ denote the decreasing permutation in $\mathfrak S_n$. Then
$(w_\circ\tau)(i)=n+1-\tau(i)$.  The recursive description just obtained says
precisely that the set $\{w_\circ\tau:\tau\in\mathfrak T_n\}$ is the set
$\mathfrak J_n$ of \dfn{Jacobi permutations} in $\mathfrak S_n$, as discussed by Petersen and
Zhuang \cite{PZ}. It follows from \eqref{eq:T-ret} that
\[
        \widetilde I_n(-1,t)
        =
        t^{n-1}
        \sum_{\gamma\in\mathfrak J_n}t^{-\operatorname{des}(\gamma^{-1})}.
\]

Let 
\[
J_n(t):=\sum_{\gamma\in\mathfrak J_n}t^{\operatorname{des}(\gamma^{-1})}.
\]
Petersen and Zhuang \cite{PZ} showed that the so-called \emph{zig-zag Eulerian polynomial} $Z_n(t)$ satisfies \[Z_n(t)=tJ_n(t)=\sum_{w\in\operatorname{Alt}_n}
        t^{\operatorname{des}_1(w^{-1})+1};
\]
they also proved that the $J_n(t)$ is palindromic of degree $n-2$.  Therefore,
\[
        \widetilde I_n(-1,t)
        =
        t^{n-1}J_n(1/t)
        =
        tJ_n(t)
        =
        Z_n(t)
        =
        \sum_{w\in\operatorname{Alt}_n}
        t^{\operatorname{des}_1(w^{-1})+1},
\]
as desired.
\end{proof} 

\section{Centralizers in the Plactic Monoid}\label{sec:plactic}  
Given a word $w$, we are interested in the Robinson--Schensted--Knuth (RSK) insertion tableau $P(w)$, which is a semistandard Young tableau. Two words $v,w$ are \dfn{Knuth equivalent}, written $v\equiv w$, if one can be obtained from the other by a sequence of \dfn{Knuth relations}, which state that for $u,v\in\mathbb N^*$ and $a,b,c\in\mathbb N$, we have
\[
uacbv\equiv ucabv \quad \text{if}\quad a\leq b<c\] and \[ubacv\equiv ubcav \quad \text{if}\quad a<b\leq c.
\]
Knuth proved that $v\equiv w$ if and only if $P(v)=P(w)$ \cite{Knu70}. The \dfn{plactic monoid} is the quotient of the free monoid $\mathbb N^*$ by Knuth equivalence; it was introduced in this form by Lascoux and Sch\"utzenberger \cite{LS81}. We use standard facts about RSK and jeu-de-taquin; convenient references are the works of Sagan \cite{Sag01,Sag20} and Stanley \cite{Sta24}.

For a word $u$, Sagan and Wilson \cite{SW25} defined its \dfn{plactic centralizer} to be the set
\[
C(u)=\{w\in \mathbb N^*: uw\equiv wu\}=\{w\in \mathbb N^*: P(uw)=P(wu)\}.
\]
They used this object to study when certain cut-and-paste operations on reading words preserve a Knuth class, and they established several necessary conditions and examples before posing conjectures. Later work of Sagan and Zhao \cite{SZ25} considers a related stability problem for plactic centralizers of powers of a fixed word.  

Rows of tableaux are numbered from top to bottom. If $T$ is a semistandard Young tableau, $\rw(T)$ is its row word, read from left to right along each row, starting with the bottom row and moving upward. Then $P(\rw(T))=T$. For a positive integer $m$, let $T_{\leq m}$ be the subtableau of $T$ consisting of entries at most $m$, and let $T_{>m}$ be the complementary skew tableau. 

We will use the following result due to Sagan and Wilson. 

\begin{lemma}[{\cite[Corollary~3.2]{SW25}}]\label{lem:SW} 
Let $u\in\mathbb N^*$, and let $w\in C(u)$. If a letter $b$ does not appear in $u$, then no $b$ can be bumped while forming $P(wu)$ from $P(w)$ by row insertion. 
\end{lemma} 

We will also need the following classical theorem due to Greene. 

\begin{theorem}[{\cite{Gre74}}]\label{thm:Greene} 
Let $w$ be a word, and let $\lambda=(\lambda_1,\ldots,\lambda_\ell)$ be the shape of $P(w)$. For every $k\geq 1$, the maximum size of a subword of $w$ that can be decomposed into $k$ weakly increasing subwords is $\lambda_1+\cdots+\lambda_k$. Let $\lambda'=(\lambda_1',\ldots,\lambda_{\ell'}')$ be the conjugate of $\lambda$. For every $k\geq 1$, the maximum size of a subword of $w$ that can be decomposed into $k$ strictly decreasing subwords is $\lambda'_1+\cdots+\lambda'_k$. 
\end{theorem} 

\subsection{Bounding Entries}

The following theorem was conjectured by Sagan and Wilson~\cite{SW25}. 

\begin{theorem}\label{thm:plactic1} 
Let $u$ be a nonempty word, and let $m$ be the maximum letter in $u$. Let $\ell$ be the number of rows of $P(u)$. If $w\in C(u)$, then all entries in the first $\ell$ rows of $P(w)$ are at most $m$. 
\end{theorem}

\begin{proof}
Fix $w\in C(u)$, and let $P(w)=T$. Suppose by way of contradiction that row $r$ of $T$ contains an entry larger than $m$, where $1\leq r\leq \ell$. Let $h$ be the number of entries in row $r$ of $T$ that are at most $m$. 

Since $w\in C(u)$, we have $w\in C(u^N)$ for every $N\geq 1$. Every entry larger than $m$ is absent from $u^N$, so \cref{lem:SW} implies that no entry larger than $m$ can be bumped while forming $P(wu^N)$ by inserting $u^N$ into $T$.

All letters inserted from $u^N$ are at most $m$. Since row $r$ contains an entry larger than $m$, any letter from $[m]$ inserted into row $r$ cannot be appended to the end of the row, so it must bump some entry. We have established that this bumped entry cannot be larger than $m$. This shows that the number of entries at most $m$ in row $r$ remains fixed during the insertion of $u^N$. Hence, the number of entries in row $r$ of $P(wu^N)$ that are at most $m$ is exactly $h$ for every $N\geq 1$.

By the standard restriction property of RSK,
\[
P(wu^N)_{\leq m}=P(w_{\leq m}u^N),
\]
where $w_{\leq m}$ denotes the subword of $w$ consisting of the entries that are at most $m$. Therefore the $r$-th row of $P(w_{\leq m}u^N)$ has length $h$ for every $N$.

Because $P(u)$ has $\ell$ rows, we know by \cref{thm:Greene} that there is a strictly decreasing subword of $u$ of length $\ell$. Therefore, $w_{\leq m}u^N$ contains $N$ disjoint strictly decreasing subwords of length $\ell$. Let $\lambda^{(N)}$ be the shape of $P(w_{\leq m}u^N)$. According to \cref{thm:Greene}, we have
\[
(\lambda^{(N)})'_1+\cdots+(\lambda^{(N)})'_N\geq \ell N. 
\]
But all entries of $P(w_{\leq m}u^N)$ are at most $m$, so each column has height at most $m$. Since the $r$-th row has length $h$, only the first $h$ columns can have height at least $r$. Thus, for $N\geq h$, we have
\[
\ell N\leq (\lambda^{(N)})'_1+\cdots+(\lambda^{(N)})'_N
\leq hm+(N-h)(r-1).
\]
This is impossible when $N$ is sufficiently large because $r\leq \ell$. 
\end{proof} 

\subsection{Evacuation} 
Define the \dfn{$m$-reverse complement} of a word $w=w_1\cdots w_n\in[m]^*$ to be the word 
\[
\RC_m(w)=(m-w_n+1)\cdots(m-w_1+1).
\]
For a tableau $T$ with entries in $[m]$, define
\[
\epsilon_m(T)=P(\RC_m(\rw(T))).
\]
This $\epsilon_m$ is the \dfn{$m$-evacuation} map used by Sagan and Wilson \cite{SW25}. Greene's theorem implies that $\epsilon_m(T)$ has the same shape as $T$; see \cite[Lemma~6.4]{SW25}. For an arbitrary tableau $T$, define $\tau_m(T)$ by replacing $T_{\leq m}$ by $\epsilon_m(T_{\leq m})$ and leaving the skew tableau $T_{>m}$ fixed. Since $\epsilon_m$ preserves shape, this replacement is well defined. 

Sagan and Wilson \cite{SW25} conjectured that 
\begin{equation}\label{eq:SW_Conj_6.5}
P(C(\RC_m(u)))=\tau_m(P(C(u))) 
\end{equation} for every $u\in [m]^*$. 
We will prove this below in \cref{thm:6.5}. We first need some lemmas. 

The first lemma we need states that $\RC_m$ is an involutive anti-automorphism of the plactic monoid restricted to words in $[m]^*$. 
\begin{lemma}\label{lem:rc}
For $x,y\in[m]^*$, we have $\RC_m(xy)=\RC_m(y)\RC_m(x)$. Moreover, $x\equiv y$ if and only if $\RC_m(x)\equiv \RC_m(y)$.
\end{lemma}

\begin{proof}
This is immediate from the definition of $\RC_m$ and the definitions of Knuth relations. 
\end{proof}

\begin{lemma}\label{cor:small}
Let $u,w\in[m]^*$, and assume $w\in C(u)$. We have $\RC_m(w)\in C(\RC_m(u))$ and $\epsilon_m(P(w))\in P(C(\RC_m(u)))$. 
\end{lemma}

\begin{proof}
Because $uw\equiv wu$, \cref{lem:rc} gives that 
\[
\RC_m(w)\RC_m(u)\equiv \RC_m(u)\RC_m(w),
\]
which is the first assertion. The second follows from the definition of $\epsilon_m$ because $\rw(P(w))\equiv w$.
\end{proof}

Let $A$ be a straight semistandard tableau with entries in $[m]$, and let $B$ be a skew semistandard tableau with entries greater than $m$. We consider the union $A\sqcup B$ only when the shapes of $A$ and $B$ are disjoint so that this union is a semistandard Young tableau. Otherwise, $A\sqcup B$ is not defined. 

\begin{lemma}\label{lem:barrier}
Let $U$ be a semistandard tableau, and let $x$ be a word in $[m]^*$. If $P(\rw(U_{\leq m})\,x)\sqcup U_{>m}$ is defined, then
\[
P(\rw(U)\,x)=P(\rw(U_{\leq m})\,x)\sqcup U_{>m}.
\]
If $P(x\,\rw(U_{\leq m}))\sqcup U_{>m}$ is defined, then
\[
P(x\,\rw(U))=P(x\,\rw(U_{\leq m}))\sqcup U_{>m}.
\]
\end{lemma}

\begin{proof}
If $P(\rw(U_{\leq m})\,x)\sqcup U_{>m}$ is defined, then entries in $U_{>m}$ do not move when we insert $x$ into $U$. This implies the first statement.  

For the second statement, compute $P(x\,\rw(U))$ by placing $P(x)$ southwest of $U$ and rectifying by jeu-de-taquin, as in the standard description of the plactic product. Compare this with the same rectification using only $P(x)$ and $U_{\leq m}$. Until an entry larger than $m$ moves, the two computations have the same holes and the same entries from $[m]$ around those holes. If a hole has both a candidate from $[m]$ and a candidate larger than $m$, the candidate in $[m]$ is chosen because it is smaller. Thus, the first possible movement of an entry larger than $m$ would have to occur when a hole has no small candidate from $[m]$ but has a candidate larger than $m$ to its right or below.
But in the rectification with only entries from $[m]$, that same hole is then an outer corner and the slide path terminates there. This cell is then deleted. Then the final tableau obtained by keeping $U_{>m}$ fixed would have a gap, contradicting the assumption that $P(x\,\rw(U_{\leq m}))\sqcup U_{>m}$ is a semistandard Young tableau. We conclude that no entry larger than $m$ moves, which completes the proof. 
\end{proof}

We can now prove \eqref{eq:SW_Conj_6.5}.  

\begin{theorem}\label{thm:6.5} 
Let $u$ be a word in $[m]^*$. We have
\[
P(C(\RC_m(u)))=\tau_m(P(C(u))).
\]
\end{theorem}

\begin{proof}
We first prove that \[
\tau_m(P(C(u)))\subseteq P(C(\RC_m(u))).
\]
Take $T\in P(C(u))$. Since $C(u)$ is a union of Knuth equivalence classes, we may use the representative $\rw(T)$. Thus,
\[
P(u\,\rw(T))=P(\rw(T)\,u).
\]
By the standard restriction property of RSK, 
\[\begin{aligned}
P(u\,\rw(T_{\leq m}))&=P(u\,\rw(T))_{\leq m} \\ &=P(\rw(T)\,u)_{\leq m} \\ &=P(\rw(T_{\leq m})\,u),
\end{aligned} 
\]
so $T_{\leq m}\in P(C(u))$.

Since $u$ has no letters larger than $m$, \cref{lem:SW} implies that no entry of $T_{>m}$ is bumped while forming $P(\rw(T)\,u)$ from $T$. Consequently, 
\[
P(\rw(T)\,u)=P(\rw(T_{\leq m})\,u)\sqcup T_{>m}.
\]

By \cref{cor:small}, $\epsilon_m(T_{\leq m})\in P(C(\RC_m(u)))$. In other words,
\[
P(\rw(\epsilon_m(T_{\leq m}))\RC_m(u))
=
P(\RC_m(u)\rw(\epsilon_m(T_{\leq m}))).
\]
Using this, we deduce that 
\[
P(\rw(\epsilon_m(T_{\leq m}))\RC_m(u))
=
\epsilon_m(P(\rw(T_{\leq m})\,u)).
\]
Indeed, $\rw(\epsilon_m(T_{\leq m}))\equiv \RC_m(\rw(T_{\leq m}))$, and \cref{lem:rc} tells us that $\RC_m$ sends the plactic class of $u\,\rw(T_{\leq m})$ to the plactic class of $\RC_m(\rw(T_{\leq m}))\RC_m(u)$.

The tableau $\epsilon_m(P(\rw(T_{\leq m})\,u))$ has the same shape as $P(\rw(T_{\leq m})\,u)$. If $A$ is any tableau of the same shape as $P(\rw(T_{\leq m})\,u)$ with entries in $[m]$, then $A\sqcup T_{>m}$ is a well-defined semistandard Young tableau (because $P(\rw(T_{\leq m})\,u)\sqcup T_{>m}$ is a well-defined semistandard Young tableau with entries in $[m]$). Hence, $\epsilon_m(P(\rw(T_{\leq m})\,u))\sqcup T_{>m}$ is a well-defined semistandard Young tableau. 

Now set
\[
U=\tau_m(T)=\epsilon_m(T_{\le m})\sqcup T_{>m}
\quad\text{and}\quad
x=\RC_m(u).
\]
Then $U_{\le m}=\epsilon_m(T_{\le m})$ and $U_{>m}=T_{>m}$. We have already shown that
\[
P(\operatorname{rw}(U_{\le m})\,x)
=
P(\operatorname{rw}(\epsilon_m(T_{\le m}))\RC_m(u))
=
\epsilon_m(P(\operatorname{rw}(T_{\le m})\,u)).
\]
Furthermore, since $\epsilon_m(T_{\le m})\in P(C(\RC_m(u)))$, we also have
\[
P(x\operatorname{rw}(U_{\le m}))
=
P(\RC_m(u)\operatorname{rw}(\epsilon_m(T_{\le m})))
=
P(\operatorname{rw}(\epsilon_m(T_{\le m}))\RC_m(u)).
\]
Therefore,
\[
P(x\operatorname{rw}(U_{\le m}))
=
P(\operatorname{rw}(U_{\le m})\,x)
=
\epsilon_m(P(\operatorname{rw}(T_{\le m})\,u)).
\]
Since
\[
\epsilon_m(P(\operatorname{rw}(T_{\le m})\,u))\sqcup T_{>m}
\]
is defined, both hypotheses of \cref{lem:barrier} hold for this choice of $U$ and $x$. Applying the two parts of \cref{lem:barrier} gives
\[
P(\operatorname{rw}(U)\,x)
=
P(\operatorname{rw}(U_{\le m})\,x)\sqcup U_{>m}
=
\epsilon_m(P(\operatorname{rw}(T_{\le m})\,u))\sqcup T_{>m}
\]
and
\[
P(x\operatorname{rw}(U))
=
P(x\operatorname{rw}(U_{\le m}))\sqcup U_{>m}
=
\epsilon_m(P(\operatorname{rw}(T_{\le m})\,u))\sqcup T_{>m}.
\]
Thus,
\[
P(\operatorname{rw}(U)\,x)=P(x\operatorname{rw}(U)).
\]
Equivalently, $\operatorname{rw}(U)\in C(x)=C(\RC_m(u))$. Since $P(\operatorname{rw}(U))=U$, this shows that
\[
U=\tau_m(T)\in P(C(\RC_m(u))).
\]
As $T$ was arbitrary, we have proved that
\[
\tau_m(P(C(u)))\subseteq P(C(\RC_m(u))).
\]

Finally, note that $\RC_m$ and $\tau_m$ are both involutions. Applying the containment just proved with $\RC_m(u)$ in place of $u$ gives
\[
\tau_m(P(C(\RC_m(u))))\subseteq P(C(u)).
\]
Applying $\tau_m$ to this inclusion yields
\[
P(C(\RC_m(u)))\subseteq \tau_m(P(C(u))).
\]
Combining the two containments proves the desired equality. 
\end{proof}

\section*{Acknowledgments} 
The author thanks Harvard University and OpenAI for providing access to ChatGPT~5.4 Pro. The author has no professional affiliation with OpenAI. 

An earlier draft of this article also presented a ChatGPT-generated proof of a formula for the Ehrhart polynomial of a partial permutohedron, which was conjectured by Behrend, Castillo, Chavez, Diaz-Lopez, Escobar, Harris, and Insko \cite{BCCSEHI}. It turns out this result was already proven earlier by Behrend \cite{Behrend}, so we have removed that section. The author thanks John Lentfer and Andr\'es R. Vindas-Mel\'endez for pointing out this reference.

\end{document}